\documentclass[11pt,a4paper,fleqn]{article}
\usepackage[german,british]{babel}
  \usepackage[linesnumbered,ruled]{algorithm2e}
\usepackage{amsthm}   

\usepackage{latexsym}
\usepackage{amsmath,amssymb}
\usepackage{epsfig,xspace,nfontdef, refdef}
\usepackage{isolatin1}
\usepackage{enumerate}
\usepackage{mathtools}

\usepackage{amscd}

\usepackage{epsfig}
 
  \usepackage{algorithmic} 

  \usepackage{eqparbox}

  \usepackage{array}  
  \usepackage{mathtools} 

\newcommand{\ignore}[1]{}

\newcommand{\vek}[1]{\mathchoice{\displaystyle\boldsymbol#1}
{\textstyle\boldsymbol#1}{\scriptstyle\boldsymbol#1}
{\scriptscriptstyle\boldsymbol#1}}
\newcommand{\mat}[1]{\mathchoice{\displaystyle\mathbf#1}
{\textstyle\mathbf#1}{\scriptstyle\mathbf#1}
{\scriptscriptstyle\mathbf#1}}

\newtheorem{lemma}{Lemma}

\newtheorem{theorem}{Theorem}
\newtheorem{proposition}[theorem]{Proposition}

 \title{Pivoted Cholesky decomposition by cross approximation for efficient solution of  kernel systems } 
\author{Dishi Liu and Hermann G.  Matthies}    

\begin{document}       

\maketitle             

\begin{abstract}

Large kernel systems are prone to be ill-conditioned. Pivoted Cholesky decomposition (PCD)  render a stable and efficient solution  to the systems  without a perturbation of regularization. This paper proposes a new PCD  algorithm by tuning Cross Approximation (CA) algorithm to kernel matrices which merges the merits of PCD and CA, and proves as well as numerically exemplifies that it solves large kernel systems  two-order  more efficiently than those resorts to regularization.    As a by-product, a diagonal-pivoted CA technique is also shown  efficient in eigen-decomposition of large covariance matrices in  an uncertainty quantification problem.

\end{abstract}

\section{Introduction}

Interpolation of spatial data with kernel functions \cite{Schaback1999} finds wide application in many fields of scientific computing. The system of equations is prone to be ill-conditioned when a large number of data or smooth kernel functions are used in hope  of improving accuracy. Nevertheless a good interpolation accuracy is usually accompanied by a large condition number of the interpolation matrix \cite{Schaback1995}.
This is like a high-wire walk between two abysses, bad accuracy on the left and numerical instability on the right, no wonder various regularization techniques \cite{Neumaier1998} are usually adopted as  safety ropes. In this paper we show pivoted Cholesky decomposition, a matrix approximation technique, is an  efficient solver in this situation which guarantees the stability without the perturbation of regularization while reduces the complexity of solution from $\C O(n^3)$ to $\C O(k^2 n)$ for a rank-$k$ system. We propose an improved pivoted Cholesky decomposition algorithm by tuning cross approximation technique to  symmetric and positive definite  matrices.
%

Pivoted Cholesky decomposition (PCD)\cite{Dongarra1979,Fine2001,Hammarling2007,Harbrecht2012} and cross approximation  \cite{bebendorf2000,bebendorf2003} are two favourable low rank approximation techniques with linear complexity in $n$ (in contrast to the cubic complexity of a truncated singular value decomposition). Compared to other  techniques of linear complexity like the Nystr\"{o}m approximation \cite{Williams2001,Drineas2005} and sparse greedy approximations \cite{Smola2000} they are more accurate due to the pivot maximizations \cite{Bach2013}. Another advantage is that they provide deterministic error bounds and  thus can be run adaptively to a certain accuracy.

 
 PCD is a pivoted version of the outer product Cholesky algorithm \cite[Algorithm 4.2.2]{Golub1996} that chooses the diagonal entry with the largest modulus   as the pivot. 
 In each step a rank-1 approximation is made based on the column and row cross at the pivot so that $k$ steps accumulate  a rank-$k$ approximation. 
 Earlier versions of this algorithm, e.g. LINPACK routine SCHDC \cite{Dongarra1979},  include a global updating  in which the rank-1 approximation is subtracted from the remainder matrix (or the original matrix for the first step), this made them more expensive than the later ones in \cite{Fine2001,Bach2003,Bach2005,Hammarling2007,Harbrecht2012}\footnote{In \cite{Bach2003,Bach2005} this method was introduced by the name of \em{incomplete Cholesky decomposition}  } in which the updating only occurs in the pivoted columns.

When cross approximation (CA) is applied to symmetric positive  definite matrices it produces similar results as a PCD, i.e. the results are only a row permutation away from  triangular matrices. If the pivot is chosen only on the diagonal, the fully pivoted CA and the early version of PCD are nearly identical with the only difference in the permutation. Such a CA was proposed in \cite{Savostyanov2009} to produce a diagonalised decomposition.  While choosing the pivot on the diagonal in PCD is justified in \cite{Bach2005} as maximising the lower bound of gain in each step,  it is not yet fully justified in the language of CA.  The CA algorithm, with the merit of being simpler (without row swapping and nested entry indexing) and providing sharp error bound in uniform norm,  can be adapted to a PCD if this gap is bridged.       
 
In this work we first bridge this gap and merge the merits of the two regimes into a new PCD algorithm which is simpler and gives sharp uniform-norm error bound, and then justify the validity and efficiency advantage of the algorithm in solving ill-conditioned kernel systems. 

 The rest of this paper is organized as follows. In Section~\ref{sec:CA}  we recall the basic Cross Approximation algorithm and introduce an adaption to symmetric positive definite matrices. In Section~\ref{sec:PCD} we extend the adapted CA to PCD.  In Section~\ref{sec:kernel} we justify the PCD solution to rank deficient kernel systems. Section~\ref{sec:applications} gives applications of the algorithms. Section~\ref{sec:summary} summarizes the whole paper.

 \section{Cross approximation} \label{sec:CA}
Cross approximation (CA) is   an iterative matrix approximation technique  \cite{bebendorf2000,bebendorf2003} that yields a factorization equivalent to a skeleton decomposition \cite[Lemma 3]{bebendorf2000}.
Like in many other methods,  CA constructs a rank-$k$ approximation by using only a small portion, i.e.  some $k$ rows and $k$ columns, of the matrix.  Hence it is widely used for data compression \cite{Litvinenko2013,Espig2012}. It is justified in \cite{goreinov2001} that  the   columns and rows should be such chosen that their intersection has the largest determinant   in modulus among all $k \times k$ submatrices (the \textit{maximal-volume principle}), and it is shown in \cite{goreinov2011}   that such an approximation is quasi-optimal in the uniform norm.   
 
 In each step of CA, a rank-1 approximation is made and subtracted  from the remainder matrix. It turns out that  choosing the pivot  (the intersection of the chosen column and row that form the rank-1 approximation) as the entry with the largest modulus in the remainder matrix is the optimal strategy in term of maximal-volume principle if we keep the pivot  points in the previous steps fixed (like in CA) \cite[Lemma 2]{bebendorf2000} .

For very large matrices, CA with global pivot searching and global updating  (which is termed as fully pivoted CA) is expensive  or even prohibitive. A trade-off of accuracy and cost is made in the so-called partially pivoted CA which only searches for the pivot in some chosen rows and columns. 

However, if the underlying matrix is symmetric positive semi-definite (SPSD) much of the cost of a fully pivoted CA can be saved without sacrificing accuracy. This is due to  the fact that the diagonal entries of SPSD matrices always include the maximum in modulus, and that the remainder matrices   can be kept SPSD during the CA steps. 
 
A SPSD matrix-adapted  CA is proposed in this section. This algorithm has a complexity linear in the matrix size  as does the partially pivoted CA, while retains the accuracy of a fully pivoted CA.   We start by introducing the ``baseline'' ---  the  fully pivoted CA algorithm.

\subsection{Notation}

Matlab-like notation is used, i.e. a matrix $\mat A$'s   $i$-th column and $j$-th row are written as $\mat A[:,i]$ and $\mat A[j,:]$ respectively, but for an identity matrix these are written as  $e_i$ and $e_j^\top$. We use the same notations for different sizes wherever  the actual size is clear by
the context. 
 
 \subsection{The basic cross approximation algorithm}
 CA algorithms produce a factorization $\mat A \mat B$ approximating a matrix $\mat M$, i.e. $ \mat A \mat B \approx \mat M$.   Fully pivoted CA (as detailed in Algorithm \ref{a1}) is the most basic, accurate and also the most expensive  one since it makes global pivot searching and global updating of remainder matrix. It yields a    factorization $\mat A \mat B$ with a user specified rank $k$ or a   maximum entry-wise error $\varepsilon$.  
 \begin{algorithm}[h!] 
    \SetKwInOut{Input}{Input}
    \SetKwInOut{Output}{Output}
   \caption{Fully pivoted Cross Approximation}
     \Input{$\mat M \in \D R^{m \times n}$,  $k \in \D Z,  1 \leq k\leq \min(m,n)$ (or $\epsilon_{tol} \in \D R $ for an adaptive version) }
     \Output{$\mat A \in \D R^{m \times k} $, $\mat B \in \D R^{k \times n} $ and $\varepsilon \in \D R$ such that $| \mat M-\mat A \mat B |_\infty \leq \varepsilon$}
  \begin{algorithmic}[1]\label{a1}
     \medskip
     
     \STATE $ \mat R_1:= \mat M, \; \ell :=1$
     
     \WHILE{ $\ell \leq k$  and $\varepsilon > 0$ (or  \bf while $\varepsilon >  \epsilon_{tol})$ }  
     \STATE $i_\ell, j_\ell := \mbox{argmax}_{i,j} \left |\mat R_\ell[i,j]\right  |$
     \STATE $  \gamma_\ell :=  \mat R_\ell[i_\ell, j_\ell]  $    , $\varepsilon := \gamma_\ell$
    \STATE $ \mat A[:, \ell] :=  \mat R_\ell[:, j_\ell]  $
      \STATE   $ \mat B[\ell, :] :=   \mat R_\ell[ i_\ell, :] / \gamma_\ell$
     \STATE $ \mat R_{\ell+1} := \mat R_{\ell} - \mat A[:, \ell] \mat B[\ell, :]   $    
     \STATE $\ell := \ell+1$
     \ENDWHILE 
 \end{algorithmic}   
 \end{algorithm}
 
The following Lemma shows the  $\mat A \mat B$ equals to a  pseudo-skeleton decomposition \cite{goreinov1997} of $\mat M$.   
\begin{lemma}\cite[Lemma 4.6]{boerm2003}   \label{lemma0} Given  matrices $\mat A $ and $\mat B$ as obtained in  Algorithm \ref{a1} approximating $\mat M$,   where $\vek i = \{i_\ell \}_{\ell=1}^k$ and  $\vek j = \{j_\ell \}_{\ell=1}^k$ collect the  pivot  indices $i_\ell$ and $j_\ell$ in Algorithm \ref{a1}, then
  \begin{align}
    \mat A \mat B =  \mat M[:, \vek j] \cdot \mat M[\vek i, \vek j]^{-1} \cdot \mat M[\vek i, :] 
 \end{align}
 
\end{lemma}

The determinant of the submatrix $\mat M[\vek i, \vek j]$ can be computed  conveniently   by 
  \begin{align}
   \det ( \mat M[\vek i, \vek j]) = \prod_{\ell=1}^k \gamma_\ell   \label{e1}
 \end{align} with $\gamma_\ell  $ as defined in Algorithm~\ref{a1}.
 
The maximal-volume principle \cite{goreinov2001} suggests that the optimal pivoting strategy is to choose $\vek i$ and $\vek j$ so that the determinant of  $ \mat M[\vek i, \vek j]$ is   maximal in modulus among all $k \times k$ submatrices. So by (\ref{e1}) maximizing $\gamma_\ell$ is a ``greedy'' strategy which maximizes the gain in each step.

This fully pivoted CA algorithm is of  complexity  $\C O(kmn)$ which is quadratic in matrix size.  A linear complexity $\C O(k^2(m+n))$  can be achieved by the partially pivoted CA   which searches only submaximal $\gamma_\ell$'s in the last pivoted row and column. This generally would sacrifice the accuracy of approximation.

 \subsection{Tuned to symmetric positive semi-definite  matrices}
A linear complexity can also be achieved without sacrificing accuracy if $\mat M$  is a  symmetric positive semi-definite (SPSD) matrix. This is due to the  fact that a SPSD matrix always have a maximum in modulus in the diagonal, and that the remainder matrix $\mat R$ can be kept SPSD during the CA process. We start from some well known facts as stated in the following two lemmas.
\begin{lemma}  \cite[p.147]{Golub1996} \label{lemma1}
If the matrix $\mat M$ is SPSD, then 
\begin{enumerate}[(i)]
 \item its diagonal entries are non-negative. 
 \item the global maximum in modulus is on the diagonal. 
\end{enumerate} 
\end{lemma}
%
 \begin{lemma} \label{lemma3}  \cite[Theorem 7.2.5]{Horn2012}
   \cite[p.40]{murota2003} 
If all principal minors of a symmetric matrix are nonnegative (positive semi-definite), the matrix is positive semi-definite.
\end{lemma} 
And we prove the positive semi-definiteness of the remainder matrix $\mat R$ as follows.  
 \begin{proposition}  \label{p1}
 In approximating a SPSD matrix $\mat M$,  if CA chooses the pivot in the diagonal, the remainder matrix remains SPSD.
\end{proposition}
 
\begin{proof}   Suppose a diagonal entry $p$ is used as the pivot,  and $\mat X, \mat Y, \mat Z$ and $\vek u, \vek v$ are  submatrices and vectors in $\mat M$. The remainder matrix $\mat R$ after the first step is
 \begin{align}
    \mat R = \mat M - \mat A \mat A^\top= & \begin{pmatrix} \mat X & \vek v & \mat Z^\top \\ \vek v^\top & p &\vek u^\top\\ \mat Z &  \vek u & \mat Y   \end{pmatrix}
  -\frac{1}{p}
 \begin{pmatrix} 
  \vek v \\ p\\ \vek u
 \end{pmatrix}
 \begin{pmatrix}
  \vek v^\top & p & \vek u^\top 
 \end{pmatrix} \nonumber \\ 
 =&  \begin{pmatrix} \mat X - \frac{1}{p} \vek v \vek v^\top& \vek 0 & \mat Z^\top-\frac{1}{p}\vek v \vek u^\top \\ 
                                      \vek 0^\top & 0 &\vek 0^\top\\ 
                                       \mat Z -\frac{1}{p}\vek u \vek v^\top &  \vek 0 & \mat Y-  \frac{1}{p} \vek u \vek u^\top  \end{pmatrix} \label{M}
 \end{align}
Because $\mat M$ is SPSD, so is its principle submatrix $\begin{pmatrix} \mat X & \vek v   \\ \vek v^\top & p    \end{pmatrix}$, and  we have the relation
  \[
  \begin{pmatrix} \mat X & \vek v   \\ \vek v^\top & p     \end{pmatrix}  =
 \begin{pmatrix} \vek I & \frac{1}{p} \vek v\\ \vek 0^\top&1\end{pmatrix}  
 \begin{pmatrix} \mat X - \frac{1}{p}\vek v \vek v^\top&\vek 0\\ \vek 0^\top & p \end{pmatrix} 
 \begin{pmatrix} \vek I &  \vek 0 \\ \frac{1}{p} \vek v^\top &1\end{pmatrix}  := \mat P^\top \mat Q   \mat P
  \]
 in which the invertibility of $\mat P$ implies that $\mat Q$ is also SPSD \cite[Observation 7.1.8]{Horn2012}. Hence  $ \mat X - \frac{1}{p}\vek v \vek v^\top$ and all its principal submatrices  are nonnegative \cite[Corollary 14.2.12]{harville1997}. And similarly this also holds for  $\mat Y-  \frac{1}{p} \vek u \vek u^\top $ and all its principal submatrices.

From Equation (\ref{M}) it is clear that all principal submatrices of $\mat R$ that are not contained in   $ \mat X - \frac{1}{p}\vek v \vek v^\top$ or $\mat Y-  \frac{1}{p} \vek u \vek u^\top $ (including $\mat R$ itself) have at least one zero column and one zero row thus have zero determinants.   Therefore all principal minors of $\mat R$ are nonnegative, so its positive semidefinitedness follows Lemma \ref{lemma1}, and its symmetry comes from the symmetry of $\mat A\mat A^\top$.  By replacing the $\mat M$ by $\mat R$ we see the SPSD property of $\mat R$ holds after every rank-1 subtraction, this proves the proposition.
   \end{proof}

 By Lemma \ref{lemma1} and Proposition \ref{p1}  we see that in approximating  ill-conditioned SPSD matrices, as the remainder matrix $\mat R$ is SPSD in every step,  it suffices for CA to do the pivot maximization \textit{only} on the diagonal.  This leads to a CA algorithm that achieves a linear complexity $\C O(k^2 n)$ in $n$ for fixed $k$ as in the partially pivoted CA,  but is as accurate as the fully pivoted one. Like in the partially pivoted CA in this new algorithm we  avoid  generating and  updating the whole matrix, also save the storage for it.    The algorithms  yields a  rank-$k$   approximation   in the form $  \mat A    \mat A^\top \approx \mat M$  and the maximum entry-wise error $\varepsilon = \| \mat M - \mat A \mat A^\top \|_{\infty}$,  as detailed in Algorithm~\ref{DCA}. Accuracy of the approximation can be controlled by using an adaptive implementation that terminates as $\varepsilon$ drops  below a given threshold $\epsilon_{tol}$.   
 \medskip
 
 \noindent {\bf Remark 1}: The  maximal entry-wise error of the algorithm as in Proposition~\ref{p1} is just the largest diagonal entry (in modulus) of the remainder matrix $\mat R$, since $\mat R$ is the error matrix and is SPSD. 
 
    \begin{algorithm}[h] \label{DCA}
     \SetKwInOut{Input}{Input}
    \SetKwInOut{Output}{Output}
    
   \caption{ Diagonal  pivoted Cross Approximation}
   
        \Input{A SPSD matrix $\mat M \in \D R^{n \times n}$ (or a function $\E M(i)$ yielding the $i$-th column of $\mat M$),  \\ $k \in \D Z,  1 \leq k\leq n$ (or $\epsilon_{tol} \in \D R $ for an adaptive version) }
     
     \Output{$\mat A \in \D R^{n \times k} $, $\varepsilon \in \D R $ and $\vek i \in \D Z^k$ such that $| \mat M-\mat A \mat A^\top |_\infty \leq \varepsilon$}
     
   \begin{algorithmic}[1]

     \medskip
     
      \STATE Initialize empty $ \mat A   \in \D R^{n \times k }$,  $\vek i \in \D R^{k}$        
    \STATE $ \vek d := \mbox{diag}(\mat M), \; \ell := 1$     
 
     \STATE $ \varepsilon := \max_i \vek d[i]$                                
     \WHILE{ $\ell \leq k$ and $\varepsilon > 0$ (or  \bf while $\varepsilon >  \epsilon_{tol})$ }  
     
      \STATE $   i_\ell := \mbox{argmax}_{i} \left | \vek d[i] \right  |$ 
      \STATE $  \gamma_\ell = \vek d[i_\ell] $ 
        \STATE     $ \mat A[:, \ell] :=  (\mat M[ :, i_\ell] - \sum_{q=1}^{\ell-1} \mat A[:, q] \mat A[q, i_\ell] )/\sqrt{\gamma_\ell}$      \\
        (or replace $\mat M[ :, i_\ell]$ by $\E M(i_\ell)$)
       
        \STATE     $ \vek d = \vek d -  (\mat A[:, \ell])^2   $ 
       \STATE $ \varepsilon :=  | \gamma_\ell|$

       \STATE $  \vek i [\ell] :=   i_\ell$
       \STATE $\ell := \ell+1$
     \ENDWHILE 
   \end{algorithmic}   
 \end{algorithm}
 
  \medskip
 \noindent {\bf Remark 2}: Since the algorithm only works with one single column of  $\mat M$ in each step, it does not require the whole matrix to be generated at once. This is especially beneficial in handling very large matrices that are beyond memory capacity, where the $\mat M$ in the input list of the  algorithm can be replaced by a function $\E M(i)$ that returns only the $i$-th column of $\mat M$.    

   
  Algorithm~\ref{DCA} produces the same result as would a fully pivoted CA(Algorithm \ref{a1}), since it just make the same global maximization  in a more efficient way by taking advantage of SPSD properties and accordingly only updates the relevant entries.  That is to say, though here only the diagonal of $\mat R$ is updated (as in step 8) in each step, once a pivot column is chosen, its ``owed'' updating is redeemed (as in step 7). So the error bound mentioned in Remark 1 also holds for Algorithm \ref{DCA} and the maximum entry-wise error is just the maximum of $\vek d$.
 
  Notice that selecting the pivot as the maximum in modulus  in the diagonal is also proved optimal in \cite{Bach2005}  in terms of maximizing the lower bound of gain in each rank-1 approximation  and in \cite{Harbrecht2012} in term of minimizing the trace norm error.

\section{Pivoted Cholesky approximations} \label{sec:PCD}
 
   In Equation (\ref{M}) we see  that in each step the CA algorithm leaves one additional zero column and zero row in the remainder matrix, so that the $j$-th column of $\mat A$ has one more zero entry than the $(j-1)$-th, hence $\mat A$ is just a  row permutation away from a triangular matrix.  Appending a row permutation to Algorithm~\ref{DCA} leads to  a pivoted Cholesky decomposition,  $\widetilde {\mat M} \approx \vek L \vek L^\top$ with $\vek L$ a $n$-by-$k$ lower triangular matrix and $\widetilde {\mat M}$ a symmetric permutation of $\mat M$ based on an index $\vek p$ yielded by the algorithm, as detailed in Algorithm~\ref{ICD} and diagrammed in the upper part of Figure~\ref{fig:PCD_alg3}.  The algorithm also produces a $k$-by-$k$ triangular matrix $\vek L_*$ that exactly reproduces the submatrix $\mat M_*$ which is the cross of the pivoted rows and columns. The reduced factorization $ \mat M_* = \vek L_* \vek L_*^\top$ is useful in solving rank deficient  systems as explained in Section~\ref{sec:kernel}.
    
The difference of this algorithm and the other PCD algorithms in \cite[p.255]{Fine2001}, \cite[p.20]{Bach2003} and  \cite{Harbrecht2012} is that this algorithm provides an entry-wise error bound $\varepsilon = \| \widetilde {\mat M} - \vek L \vek L^\top\|_\infty$, while in other algorithms the bound of error is given in terms of sum of eigenvalues. And this algorithm is relatively simpler since it involves neither row or column permutations during the iteration, nor nested entry indexing.

 
    \begin{algorithm}[h] \label{ICD}
    \SetKwInOut{Input}{Input}
    \SetKwInOut{Output}{Output}    
   \caption{Low rank pivoted Cholesky decomposition based on CA}
  
     \Input{A SPSD matrix  $\mat M \in \D R^{n \times n}$,  $k \in \D Z,  1 \leq k\leq n$ (or $\epsilon_{tol} \in \D R $ for an adaptive version)}     
     \Output{$\vek L \in \D R^{n \times k} $,   $\vek L_* \in \D R^{k \times k} $, $\varepsilon \in \D R $, $\vek i \in \D Z^k$ and $\vek p \in \D Z^n$ \\ such that $|\mat M[\vek p,\vek p]-\vek L \vek L^\top |_\infty \leq \varepsilon$ and $\mat M[\vek i,\vek i] = \vek L_* \vek L_*^\top$ }
      \begin{algorithmic}[1]
     \medskip
      
      \STATE Run Algorithm \ref{DCA} to obtain  $\mat A \in \D R^{n \times k} $, $\varepsilon \in \D R $ and $\vek i \in \D Z^k$
       \STATE $   \vek j  :=  \{ 1,2, \cdots, n\} \setminus \vek i$  \COMMENT{complement of $\vek i$}
        \STATE $  \vek p:=  [\vek i \;\; \vek j ]$    \COMMENT{concatenation of $\vek i$ and $\vek j $}
         \STATE $ \vek L := \mat A[\vek p, :] $  \COMMENT{take rows according to index $\vek p$   }
            \STATE $ \vek L_* := \mat A[\vek i, :] $  \COMMENT{take rows according to index $\vek i$   }
   \end{algorithmic}   
 \end{algorithm} 

The convergence of PCD is proved to be exponentially fast  in $k$ if the function underlies the matrix $\mat M$ has exponentially decaying eigenvalues  \cite[Theorem 2]{Harbrecht2012}. This kind  of function is not rare \cite{Schwab2006}. 

  \begin{figure}[htbp!]   
  \hspace{2cm} \includegraphics[width=6.6cm]{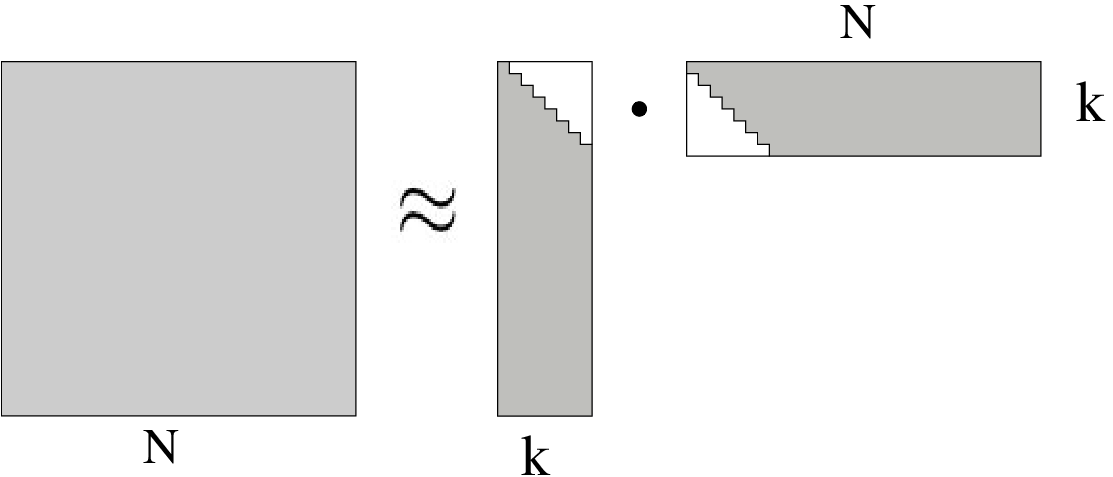}   
 \caption{Algorithm~\ref{ICD}.  Unshaded zones depict zero entries.   Upper: $\widetilde {\mat M} \approx \vek L \vek L^\top $. Lower: $ \mat M_* =  \vek L_* \vek L_*^\top $. }
 \label{fig:PCD_alg3}
 \end{figure}
 
   \begin{figure}[htbp!]    
  \hspace{2cm} \includegraphics[width=8cm]{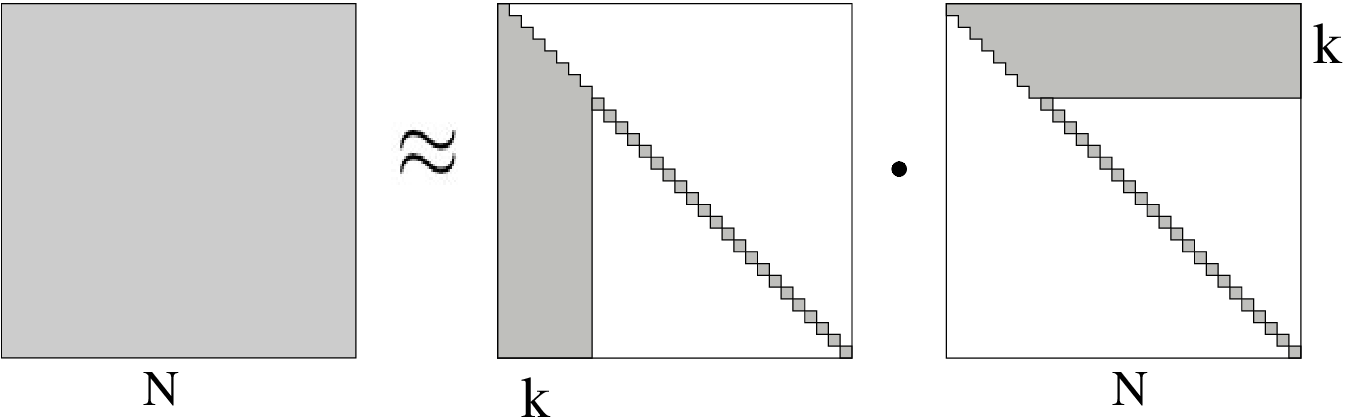}  
 \caption{Algorithm~\ref{ICD2},  $\widetilde {\mat M} \approx \vek L_n \vek L_n^\top $.  Unshaded zones depict zero entries. }
 \label{fig:PCD_alg4}
 \end{figure}
 
 The accuracy of Algorithm \ref{ICD} can be improved by a slight variation  with minor extra cost. The variant yields a full rank factorization,  $\widetilde {\mat M} \approx \vek L_n \vek L_n^\top$ with $\vek L_n$ a $n$-by-$n$ matrix, by filling the empty diagonal entries of $\vek L$ with the square root of the non-pivoted diagonal entries of $\mat R$ at the end of the procedure. This is detailed in  Algorithm \ref{ICD2} and  diagrammed in   Figure~\ref{fig:PCD_alg4}.

     \begin{algorithm}[h]      \label{ICD2}
    \SetKwInOut{Input}{Input}
    \SetKwInOut{Output}{Output}    
    
   \caption{Full rank pivoted Cholesky decomposition}
   
      \Input{A SPSD matrix  $\mat M \in \D R^{n \times n}$,  $k \in \D Z,  1 \leq k\leq n$ (or $\epsilon_{tol} \in \D R $) }     
     \Output{$\vek L_n \in \D R^{n \times n} $, $\varepsilon \in \D R $, $\vek p \in \D Z^n$  
     such that $|\mat M[\vek p,\vek p]-\vek L_n \vek L_n^\top |_\infty \leq \varepsilon$ }  
   
   \begin{algorithmic}[1]

     \medskip
      
      \STATE Run Algorithm \ref{DCA} with a slight variation that initializes $\mat A $ as an n-by-n matrix to obtain  $\mat A \in \D R^{n \times n} $, $\vek d \in \D R^n$, $\varepsilon \in \D R $, $\vek i \in \D Z^k$ and $\ell \in \D Z$
       \STATE $  \vek j :=  \{ 1,2, \cdots, n\} \setminus \vek i$  \COMMENT{complement of $\vek i$ in $\{ 1,2, \cdots, n\}$ }
       \STATE $  \vek q :=  \{ \ell, \ell+1, \cdots, n\}  $  
       \STATE $  \mat A[\vek j, \vek q]:=     \sqrt{\vek d[\vek j]}$    \COMMENT{fill the unspecified diagonal entries}
        \STATE $  \vek p:=  [\vek i \;\; \vek j]$    \COMMENT{concatenation of $\vek i$ and $\vek j$}
         \STATE $ \vek L_n:= \mat A[\vek p, :] $  \COMMENT{take rows indexed by $\vek p$   }
   \end{algorithmic}   
 \end{algorithm}
  
Apart from the exact reproduction of  the pivot columns and rows as do  all CA and PCD algorithms\cite[Lemma 4.5]{boerm2003},  Algorithm~\ref{ICD2} in addition exactly reproduces the diagonal entries of the original matrix (this can be easily seen by the all-zero diagonal of $\mat R = \widetilde {\mat M} - \vek L_n \vek L_n^\top$) while the off-diagonal part remains the same as that from Algorithm~\ref{ICD}.  $\varepsilon$ in Algorithm~\ref{ICD2} is a loose upper bound, rather than the exact one in Algorithm~\ref{ICD},  of the entry-wise error.

 
 


In case of solving rank deficient system the submatrix factorization  $\mat M_* = \vek L_* \vek L_*^\top$ is useful. If the system is consistent it is sufficient to solve a reduced system of $\mat M_*$, with only an $\C O(k^2)$ complexity  in addition to the cost of PCD. This is detailed in the next section.

\section{Solving large  kernel systems} \label{sec:kernel}

Large kernel matrices are prone to be ill-conditioned or rank-deficient, and it is also well-known that  a well-performing kernel system  usually associates with an ill-conditioned matrix. Like an echo of the \textit{Uncertainty Principle} in quantum mechanics, in \cite{Schaback1995} it was stated that the condition number and the accuracy cannot be both good.   Regularization proportional to the identity matrix  is usually the cure in this situation. But we would show here PCD can solve these systems more efficiently.

Let us take radial basis functions (RBF) as an example.  Consider modelling a function $f: \Omega \rightarrow \D R$ on some compact domain $\Omega \subseteq \D R^d$ by a linear combination  of radial basis functions $\phi: \D R_+ \rightarrow \D R$ each centered at one of the points in a scattered and distinct set $\mat X =\{\vek x^{(1)}, \cdots, \vek x^{(n)} \} \subset \Omega$:
\begin{equation}
 f(\vek x) \approx s(\vek x) = \sum_{i=1}^n w_i \; \phi_i(\vek x) ,  \;\; \mbox{with}\;\; \phi_i(\vek x) = \phi(\|\vek x - \vek x^{(i)}\|)   \label{RBF_1}
\end{equation}  
where $ \|\cdot \|$ denotes Euclidean norm. The coefficients $\vek w=\{w_1,\cdots, w_n \}$ are to be  determined by fitting $s(\vek x)$ to $n$ samples of $f$ at $\mat X$, i.e.
\begin{equation}
\vek \Phi \vek w = \vek f  \label{RBF_2},
\end{equation}  
where $\Phi_{ij} = \phi_i(\vek x^{(j)})$ and $f_i = f(\vek x^{(i)})$. The  matrix $\vek \Phi$ should be made positive definite   by a proper choice of $\phi$ and a distinct point set $\mat X$  thus (\ref{RBF_2}) has a unique solution.

However, a large condition number could render $\vek \Phi$   numerically singular. This happens often when  we have a larger $n$ and/or use smoother RBFs (all in  the hope to increase accuracy). 
In the work \cite [Theorem 5]{Wathen2015} it is shown if $\phi \in C^{\nu -1}$ and its $\nu$-th derivative is of bounded variation, then its $n$-th eigenvalue decays at least in the order $\mathcal O(n^ {-\nu-1/2})$.
This is the reason smooth radial basis functions often result  in ill-conditioned system  matrices even with a moderate $n$.  We also experimentally observed in approximating the matrices with PCD that as $n$ increases the $k$ (corresponds to number of eigenvalues above machine epsilon) would stabilize at a fixed value (which makes the complexity $\C O(k^2n)$ linear in $n$).

If the system matrix $\mat \Phi$ has only $k$ eigenvalues that are above machine epsilon, it is reasonable to use only $k$ instead of $n$  $\phi_i$ in the approximation~(\ref{RBF_1}).  The PCD in Algorithm~\ref{ICD} is the choice for this purpose. Taking $\epsilon_{tol}$ as   the machine epsilon it picks out the $k$ chosen $\phi_i$ (indexed by $\vek i$, one of its outputs) so that we can solve for the corresponding $k$ coefficients $w_i$ in~(\ref{RBF_1}) by a much reduced system using only a $[\vek i, \vek i ]$-indexed submatrix $\vek \Phi_*$: 
 \begin{equation}
\vek \Phi_* \vek w_* =  \vek L_*  \vek L_*^\top \vek w_* = \vek f_*  \label{RBF_3}
\end{equation}   
with $\vek f_*$ the  $\vek i$-indexed subset of $\vek f$, and leave the $n-k$ remaining coefficients zero.  This leads to a solution $\widetilde  {\vek w} = [\vek w_* , \vek 0 ]^\top$. This triangular system costs only $\C O(k^2)$ flops to solve. 

If  the original system  (\ref{RBF_2}), despite the rank-deficiency, is consistent, $\widehat {\vek w}$ is its solution to the machine precision. This can be seen in the following.
Let us  write the symmetrically $\vek p$-permuted $\vek \Phi$  in a block-form:
\[
        \widetilde {\vek \Phi} 
   =
    \begin{bmatrix}
      \vek \Phi_* & \mat Z^\top           \\[0.3em]
      \mat Z & \mat Y  \\[0.3em] 
     \end{bmatrix}    
\]
 and define
\[
     \vek N =  
      \begin{bmatrix}
      - \vek \Phi_*^{-1}\mat Z^\top           \\[0.3em]
      \vek I    \\[0.3em]       
     \end{bmatrix}   ,
\;\;\;
   \mat R_Y = \mat Y - \mat Z   \vek \Phi_*^{-1} \mat Z^\top .  
 \]
By Lemma \ref{lemma0}    $\mat R_Y$ vanishes as $\epsilon_{tol} \longrightarrow 0$, and $\|\mat R_Y \|_\infty = \epsilon_{tol}$.      The  inverse of $\widetilde {\vek \Phi}$ in block form is \cite{Brookes2011}\cite[p.25, Eq. 0.8.5.6]{Horn2012}\footnote{In this referenced equation, the $A_{11}$ is a typo of $A_{11}^{-1}$}: 
\begin{align}
        \widetilde {\vek \Phi}^{-1}
   =
    \begin{bmatrix}
      \vek \Phi_* & \mat Z^\top           \\[0.3em]
      \mat Z & \mat Y  \\[0.3em] 
     \end{bmatrix}^{-1}    
     & =   \begin{bmatrix}
      \vek \Phi_*^{-1} & \vek 0           \\[0.3em]
      \vek 0  &  \vek 0   \\[0.3em] 
     \end{bmatrix}             + \vek N \mat R_Y^{-1} \vek N^\top. \label{block_inverse}
\end{align}
 The second term in the right hand side is the culprit for the instability, as $\epsilon_{tol}$ vanishes, in case of rank-deficiency it overflows (and in case of full-rank the term would be null).  
 
 If we drop the second term the solution is   just  $\widetilde {\vek w}$, in this case we have
\begin{align}
        \widetilde {\vek \Phi} \widetilde  {\vek w}   =\begin{bmatrix}
      \vek f_*       \\[0.3em]
      \mat Z  \vek \Phi_*^{-1} \vek f_*   \\[0.3em] 
     \end{bmatrix} 
     \underset{\mathrm{if \,\eqref{RBF_2} \, consistent}}{=\joinrel=} \widetilde {\vek f} ,
     \label{solution}
\end{align} 
where the second equivalence holds if the original system  (\ref{RBF_2})   is consistent.
 
 This is analogical to the scenario   we solve by a  singular value decomposition (SVD) of $ \widetilde {\vek \Phi}$ and replace the reciprocals of very small eigenvalues replaced by zeros,  which produces the solution with the smallest $L_2$ norm to the underdetermined system \cite[p. 69]{press2007}.  Comparing to that, the PCD solution has a smaller complexity of $\C O(k^2 n)$ at the sacrifice of non-orthogonality of the ``selected bases''. 

For the rank-deficient systems the above solution also costs less than a Sherman-Morrison-Woodbury\cite[p.50]{Golub1996} inversion of $\lambda \vek I + \vek L_*\vek L_*^\top$ with $\lambda$ a regularization. Besides the PCD cost,  the former requires additional $\C O(k^2)$ flops for the solution while the latter requires $\C O(k^2 n)$.   
 

\section{Applications}  \label{sec:applications}
We exemplify two applications of the proposed diagonal pivoted Cross Approximation algorithm (Algorithm~\ref{DCA}) and pivoted Cholesky decomposition algorithm (Algorithm~\ref{ICD}).
 
\subsection{Eigen-decomposition of large matrices}

For aerodynamic robust design and uncertainty quantification the geometric uncertainties of aircraft are often modelled by random fields using a Karhunen-Lo\`{e}ve expansion (KLE) which requires an eigen-decomposition of the underlying covariance matrix. If the random field has a large number of nodes the matrix is often so huge that the eigen-decomposition may become  prohibitive for commonly available computing resources. In \cite{khoro2009} a  hierarchical low rank approximation technique is used to reduce the cost.  In this example we show how  the  relatively simpler diagonal pivoted cross approximation  (Algorithm~\ref{DCA}) can do the same job for KLE with continuous covariance functions.

Consider the wing surface as shown in the left part of  Figure~\ref{fig:wing} that is discretized into  56312 mesh nodes $p_i$ and  assumed   subject to zero-mean Gaussian random perturbations. Due to engineering reasons the standard deviation ($\sigma$) of the  perturbation is  $p_i$-dependent as shown in the right part of Figure~\ref{fig:wing}.  The correlation of the perturbations on any $(i,j)$ pair of node $p_k=(x_k, \,y_k,\, z_k)$ is assumed of Gaussian type:
 \[
r(p_i, p_j) = e^{(x_i-x_j)^2/\theta_x^2 + (y_i-y_j)^2/\theta_y^2 + (z_i-z_j)^2/\theta_z^2}
\]
with the correlation length $\vek \theta=(0.1,0.2,0.01)$.   This function is positive definite so that guarantees the positive definitedness of the corresponding covariance matrix.

  \begin{figure}[htbp!]  
 \centering
\includegraphics[width=5.5cm]{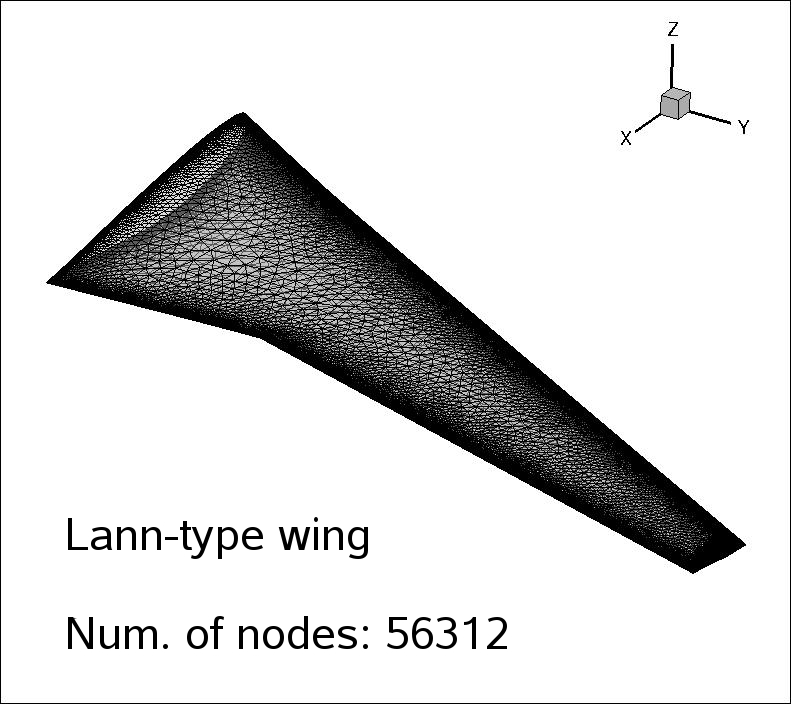}  \hspace{0.2cm}
\includegraphics[width=5.5cm]{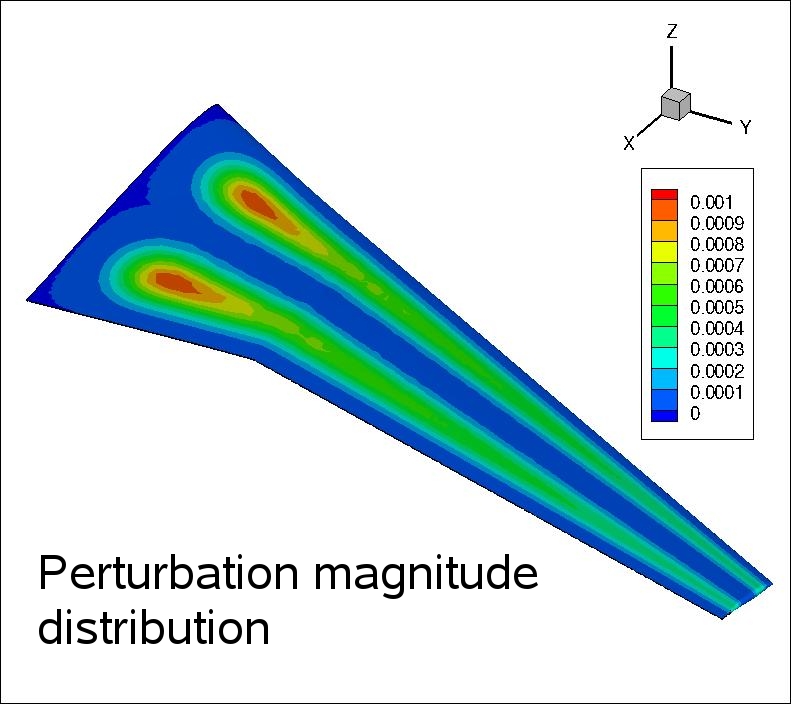} 
\caption{ Mesh of wing surface (left) and $\sigma$ distribution (right) }
\label{fig:wing}
\end{figure} 

By the above settings, a covariance matrix $\vek C$ of size $n=56312$ is composed by $c_{i,j}=\sigma(p_i) \sigma(p_j) r(p_i, p_j)$.  Suppose $\lambda_\alpha$ and $\varphi_\alpha$ are eigenvalues and eigenvectors of  $\vek C$, a parameterization and approximation of the random field  is  given by a truncated Karhunen-Lo\`{e}ve expansion (KLE): 
\begin{equation}
 \C R(x, y, z)  \approx \sum_{\alpha=1}^{k'}\xi_\alpha \;  \sqrt{\lambda_\alpha} \; \varphi_\alpha(x,y,z) \label{KLE}
\end{equation}                                                                                                 
where $\xi_\alpha$ are independent standard Gaussian random variables, and $k'$ usually much smaller than $n$.

However,   a full eigen-decomposition  is usually  not only very expensive but also not necessary, because $\vek C$ is often large (sized more than 25 GB in our case) and numerically rank-deficient due to the high degree of smoothness of the Gaussian correlation function.   We apply a diagonal pivoted cross approximation  (as in Algorithm~\ref{DCA}) to the  symmetric positive  definite matrix $\vek C$ with $k$=600 which yields an approximation $\mat A \mat A^\top \approx \vek C$ with $\mat A \in \D R^{56312 \times 600}$ associated with  a maximum entry-wise error  $\varepsilon$=1.05e-16. 

After this   approximation the eigen-decomposition can be obtained as follows. First one makes a QR decomposition   $\mat A = \mat Q_A \mat R_A$, followed by a singular value decomposition (SVD)   $\mat R_A \mat R_A^\top = \vek U \vek \Lambda \vek U^\top$. Then the diagonal of $\vek \Lambda$ contains the eigenvalues of $\mat A \mat A^\top$   and the matrix $\vek \Phi = \mat Q_A \vek U$ contains the eigenvectors. 

Notices that the complexity of the QR decomposition and of the SVD are $\C O(k^2 n)$ and $\C O(k^3)$ respectively, much smaller than the $\C O(n^3)$ complexity of a direct eigen-decomposition. On a 3.5GHz processor, the  cross approximation takes about 12 seconds  and the successive eigen-decomposition about 47 seconds, while  a direct eigen-decomposition takes about 38 hours.  This means that the low rank approximation brings a speed-up of more than 2000 times. 


With the obtained eigenpairs we implement the KLE in equation~(\ref{KLE}) with $k' = k$ which generates a random field $\C R$ parameterized by 600 Gaussian variables. Imposing the perturbation $\C R$ to the direction normal to the wing surface we obtain randomly deformed wing geometries of which Figure~\ref{fig:deform} displays three examples.

 \begin{figure}[htbp!]  
\centering
\includegraphics[width=3.8cm]{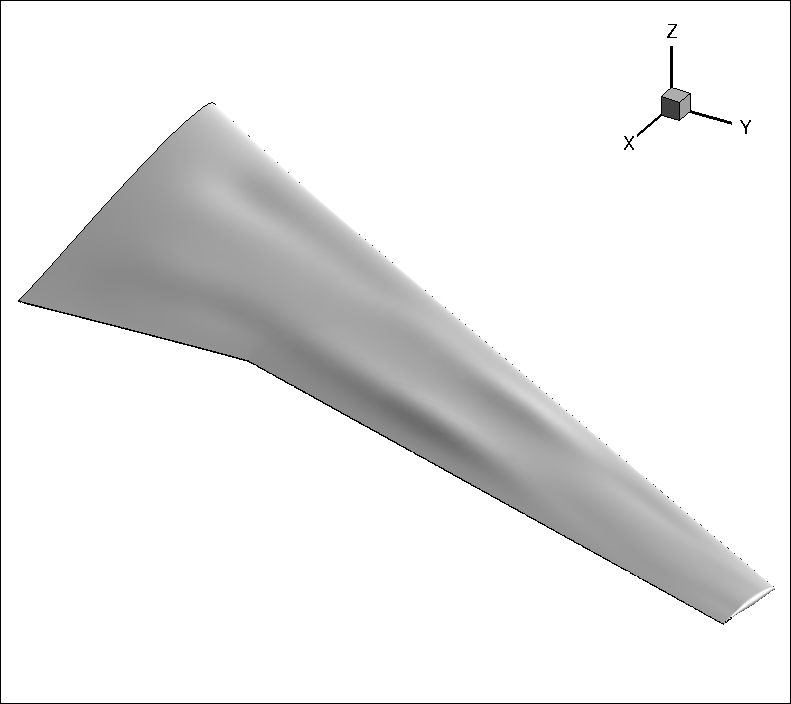}  
\includegraphics[width=3.8cm]{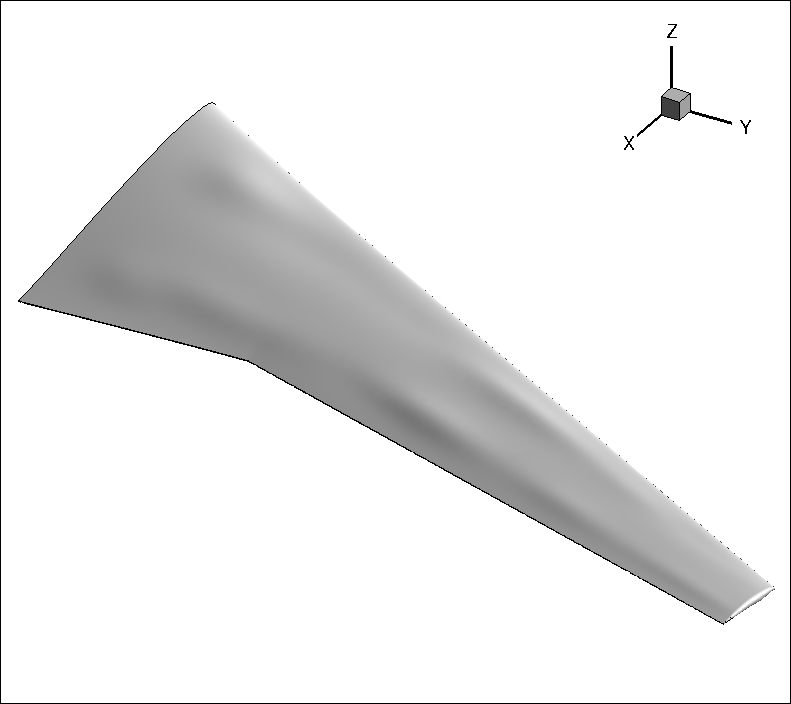}  
\includegraphics[width=3.8cm]{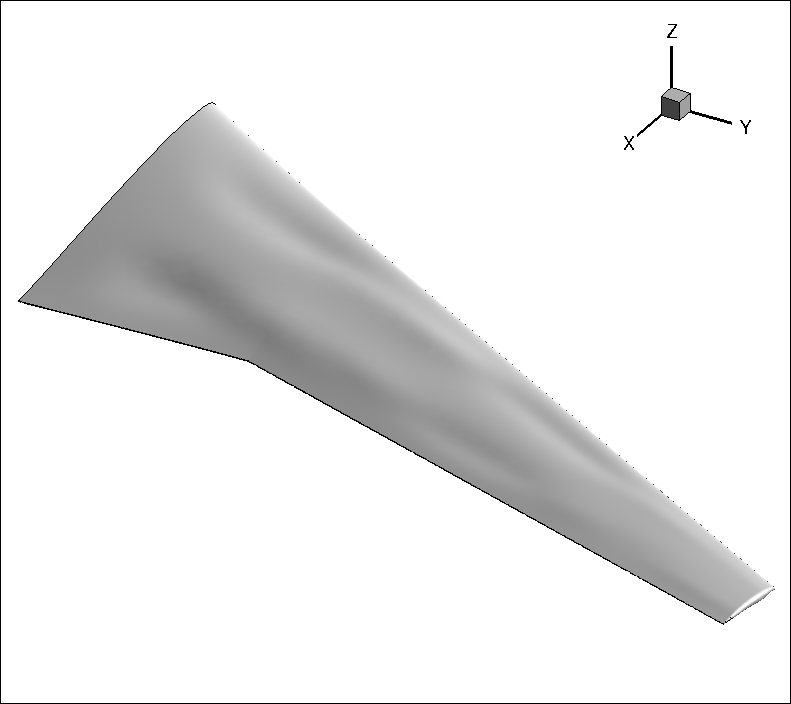}   
\caption{Three examples of randomly deformed wing (deformation three times exaggerated for illustration)}
\label{fig:deform}
\end{figure}
\subsection{Solving  radial basis functions system} 

Consider  to approximate the  following 1D and 2D test functions,
\[
 f_1 = (6x-2)^2 \sin(12x-4), \;\; x \in [0, 1]    
\]
\[
 f_2 = (1-x_1)^2  + 100(x_2-x_1^2)^2, \;\;  \vek x \in [-1, 1]^2  \;\;\;\mbox{(Rosenbrock function)}
\]
by the linear combination of radial basis functions (RBF)   in (\ref{RBF_1}) with a Gaussian type RBF
\begin{equation}
\phi(\|\vek x - \vek x^{(i)}\|) = e^{ -\|(  \vek x -  \vek  x^{(i)}) /  \vek \theta \|^2  } \label{RBF_4}
\end{equation}  
 in which $ \vek x^{(i)}$ is a sampled point and $\vek \theta $  a shape parameter. This type of RBF is superior to compactly-based ones in approximating smooth functions but prone to be ill-conditioned.
 
  \begin{figure}[htbp!]  
  \centering 
 \includegraphics[width=6.0cm]{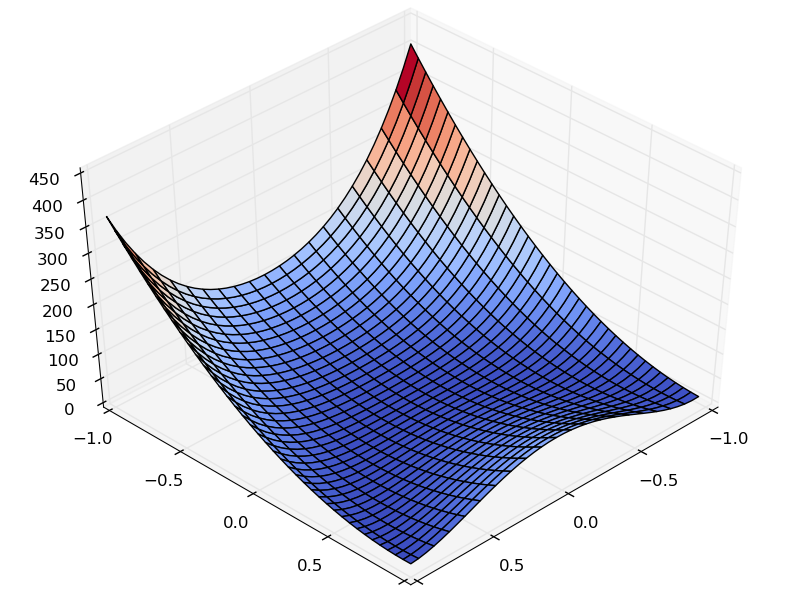}   
\caption{2D Rosenbrock function ($f_2$)}
 \label{fig:rosenbrock}
 \end{figure}

We compare the approximation based on the PCD in algorithm~\ref{ICD} (abbreviated as RBF-PCD) to that based on the regularized system (\ref{RBF_2}) (abbreviated as RBF-Chol), in their accuracy and cost. The PCD is run with an error tolerance $\epsilon_{tol}= \ell  \cdot \C E(\overline {\Phi_{ij}} )$ in which $\C E(\alpha)$ is the maximum round-off error of operations on floating-point numbers with modulus  $\alpha$ (on our machine, $\C E(1.0) \approx 2.22e$-16), and $\ell$ the number of steps the PCD algorithm executed. The solution of the equation (\ref{RBF_2}) is determined by solving the reduced system (\ref{RBF_3}) and padding zeros to the undetermined entries. In RBF-Chol the amount of regularization  is $n  \cdot \C E(\overline {\Phi_{ij}} )$.



For the 1D function $f_1$ the training points are chosen on the basis of mid-point rule, i.e. $\mat X = \{\frac{1}{2n}, \frac{3}{2n}, \cdots, \frac{2n-1}{2n} \}$. The accuracy is measured  in the root mean square error (RMSE) on 10000 test points sampled by the same rule.  We first investigate the accuracy with the shape parameter $\theta$ varies in $[0.001, 1.5]$ (which covers the optimal $\theta$ value, i.e. the one leads to the best accuracy) and with $n=50$ and $100$.  An  unregularized system (\ref{RBF_2}) (abbreviated as RBF-LU) is also included here for comparison, which is solved  by a LU decomposition since some $\theta$ values would lead to numerically non-positive definite matrices. This result is  in figure~\ref{fig:error_1D} where the RMSE is displayed on the ordinate on the left while the $k$ value, i.e. the number of utilized  samples, is displayed on the ordinate on the right.

 In figure~\ref{fig:error_1D}  we see that when $\theta$ is near its lower end   (in a well-conditioned zone), the three approximations show no difference in accuracy, i.e. on a numerically full rank system the PCD with $\epsilon_{tol}$ set to machine tolerance leads to the same solution as does a plain Cholesky decomposition.  At larger $\theta$ values  the rank is reduced and RBF-LU  displays instability while  the other two are more stable. Notice  the optimal $\theta$ (those lead to the best accuracy) are all associated with rank-deficiency, though severe rank-deficiency eventually deteriorates the accuracy. For the RBF-Chol  this deterioration is caused by the regularization which is multiplied by the very large solution $\vek w$ in this scenario, while for the RBF-PCD by decrease of utilized samples.
 
Figure~\ref{fig:curves} graphs the RBF-LU and RBF-PCD approximations with $\theta = 0.2$ and 1.0,  $n = 50$ and $100$.  Green dots   depict the  samples used(included in $\vek f_*$) by the RBF-PCD algorithm and red ones the rest. This figure displays vividly that the latter approach produces as accurate or better and stabler result by using the ``key'' samples only.   

To compare the error convergence and time cost along $n$, we first identify a series of optimal $\theta$ for a series of $n$ values, each by a fine grid search,  for  the regularized RBF-Chol and RBF-PCD respectively, and make the RBF approximations with the optimal $\theta$ for each $n$. Figure~\ref{fig:error_varied_n} and \ref{fig:error_varied_n_2D} shows the convergence of RMSE along $n$ and the associated $k$ values  for $f_1$ and $f_2$ respectively. Figure~\ref{fig:time_varied_n} contrasts the  time costs of the two approaches which are averages of 100 runs. 

It is seen the RBF-PCD has comparable or slightly better accuracy than the regularized RBF-Chol, while costs much less time. The computational time of RBF-PCD manifests its complexity $\C O(k^2n)$ which is  nearly linear in $n$ due to the eventually stabilized $k$.

  \begin{figure}[htbp]  
  \centering
  \hspace{-0.5cm}
 \includegraphics[width=7.5cm]{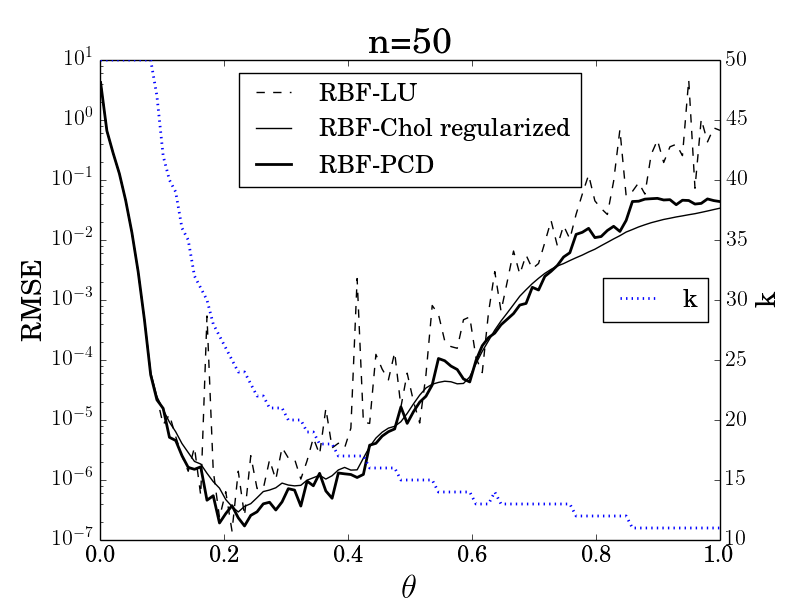}  
 \hspace{-0.2cm}
 \includegraphics[width=7.5cm]{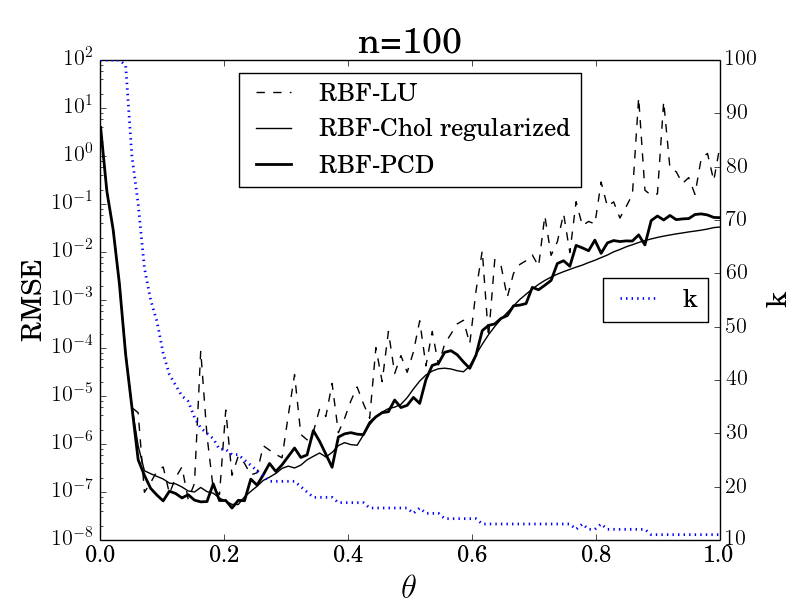}  
 \caption{Error of RBF approximations of $f_1$ along $\theta$, with $n=50$ (left), $100$ (right), and $k$ of  RBF-PCD read by the right axes}
 \label{fig:error_1D}
 \end{figure}

  \begin{figure}[htbp!]  
  \centering
  \includegraphics[width=7.3cm]{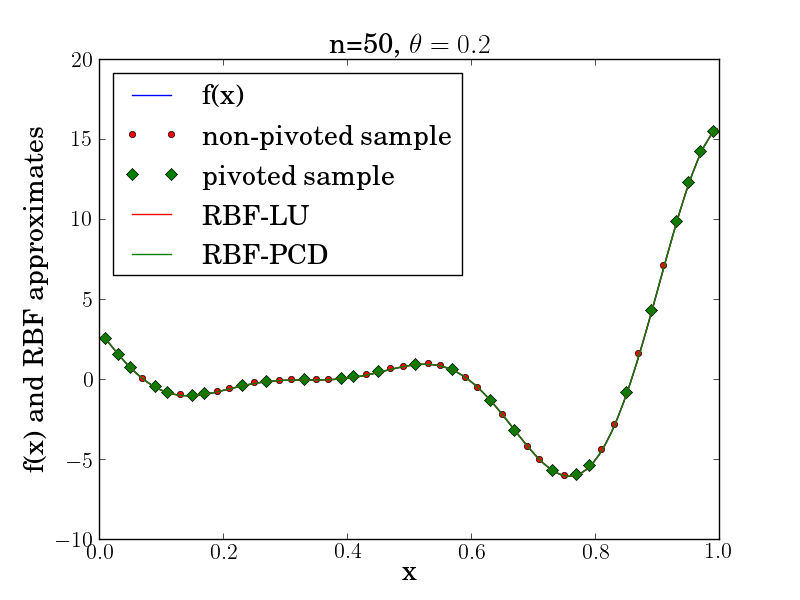}  
 \includegraphics[width=7.3cm]{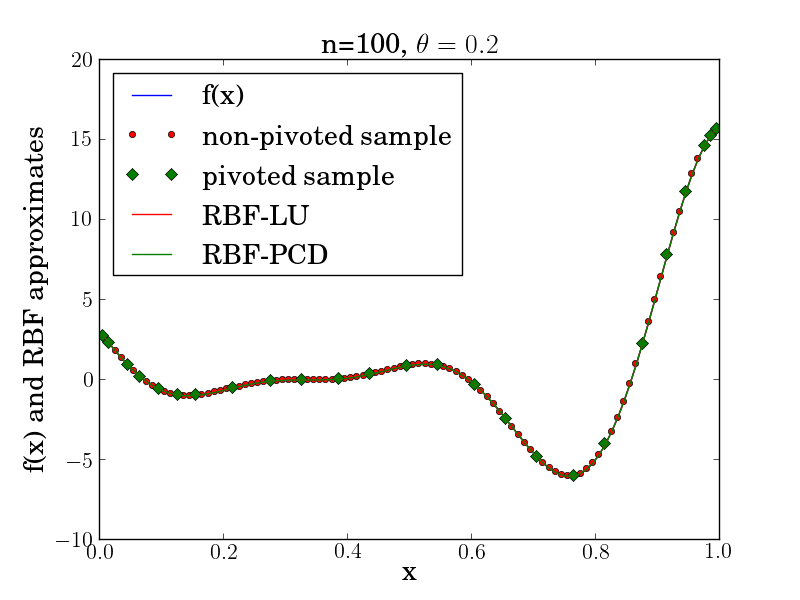}  
 \includegraphics[width=7.3cm]{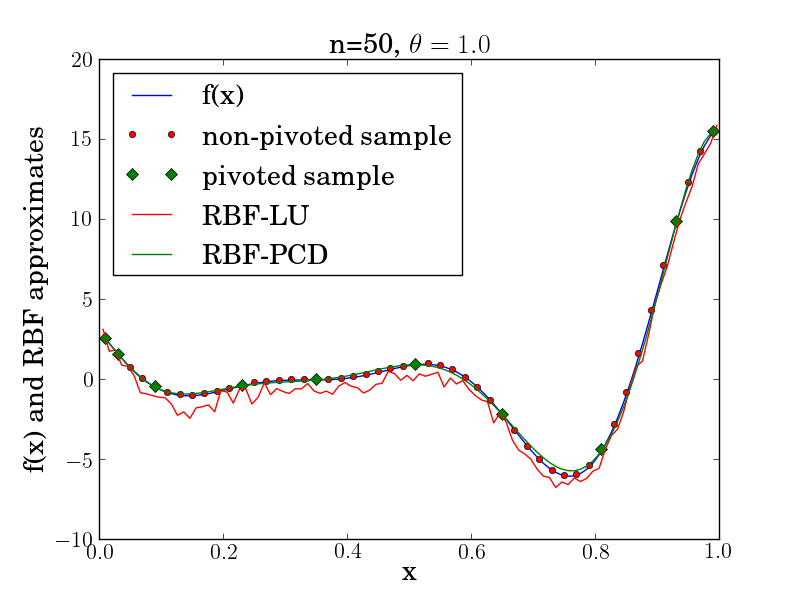}  
 \includegraphics[width=7.3cm]{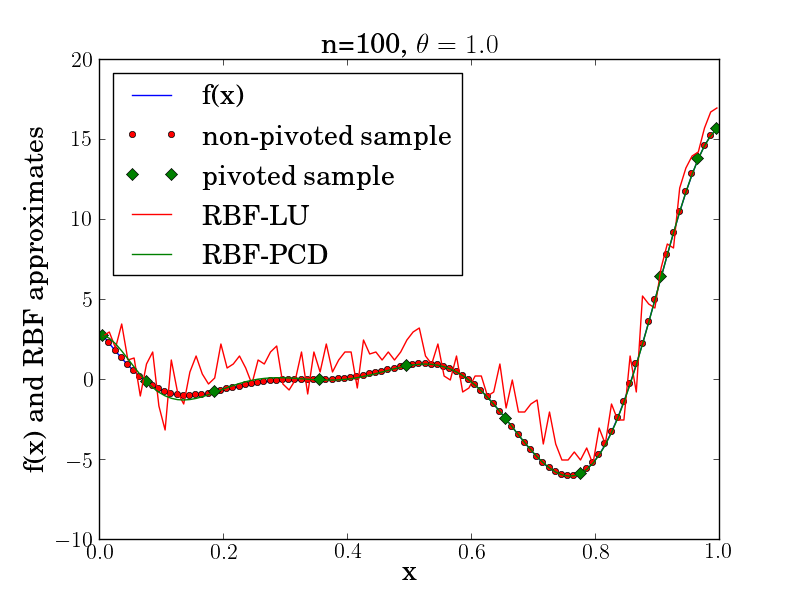}   
 \caption{Stabilizing effect of PCD in RBF approximations, with $n=50$ (left) and $100$ (right), $\theta = 0.2$ (top) and $1.0$ (bottom).}
 \label{fig:curves}
 \end{figure}

  \begin{figure}[htbp!]  
  \centering
  \hspace{-0.5cm}
 \includegraphics[width=7.5cm]{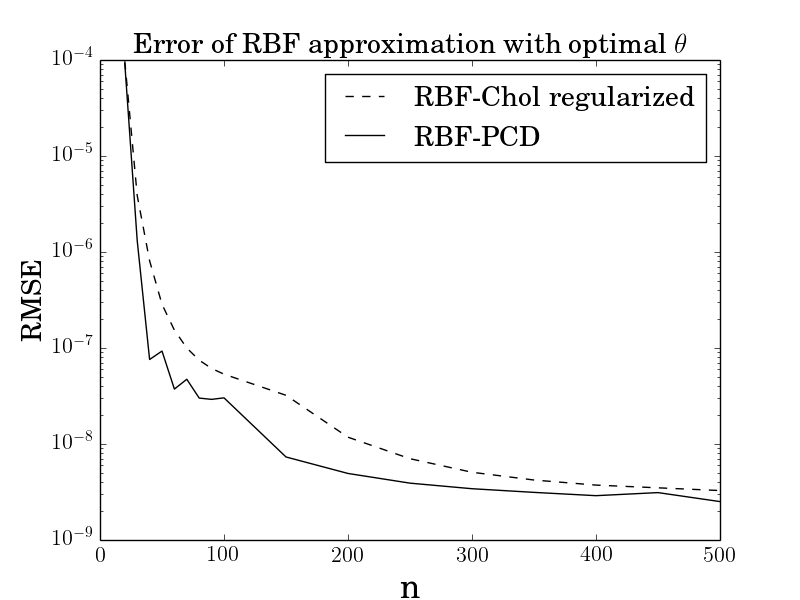}  
 \hspace{-0.2cm}
 \includegraphics[width=7.5cm]{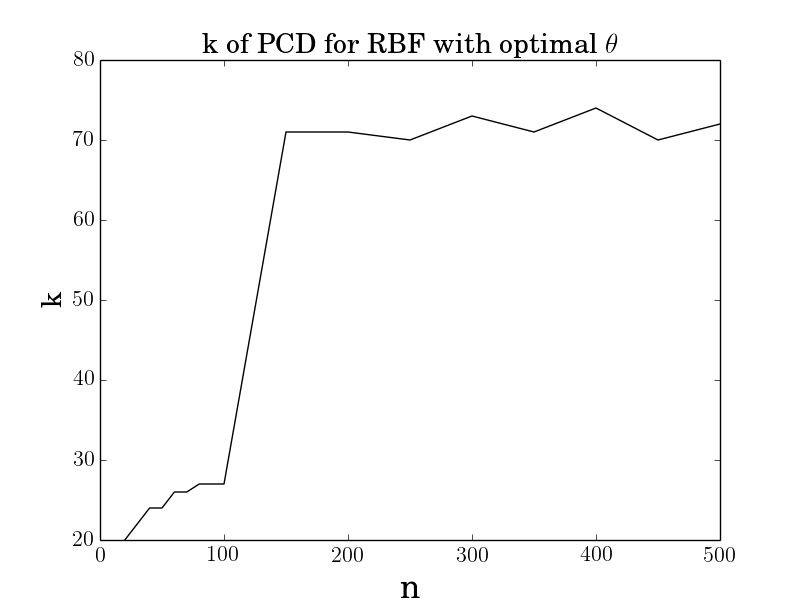}  \\   
\caption{Error of RBF approximations of $f_1$ along $n$ with optimal $\theta$ (left), and the associated $k$ of RBF-PCD (right)}
 \label{fig:error_varied_n}
 \end{figure} 
 
   \begin{figure}[htbp!]  
  \centering
  \hspace{-0.5cm}
 \includegraphics[width=7.5cm]{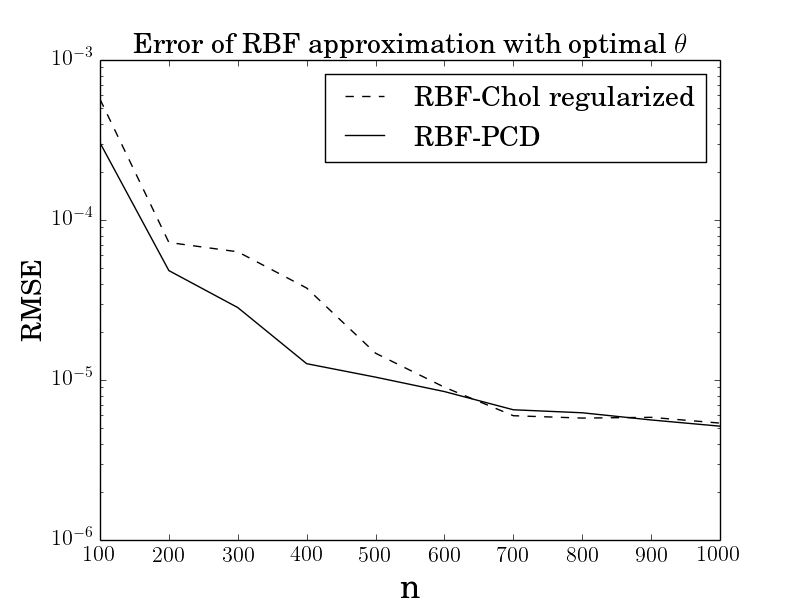}  
 \hspace{-0.2cm}
 \includegraphics[width=7.5cm]{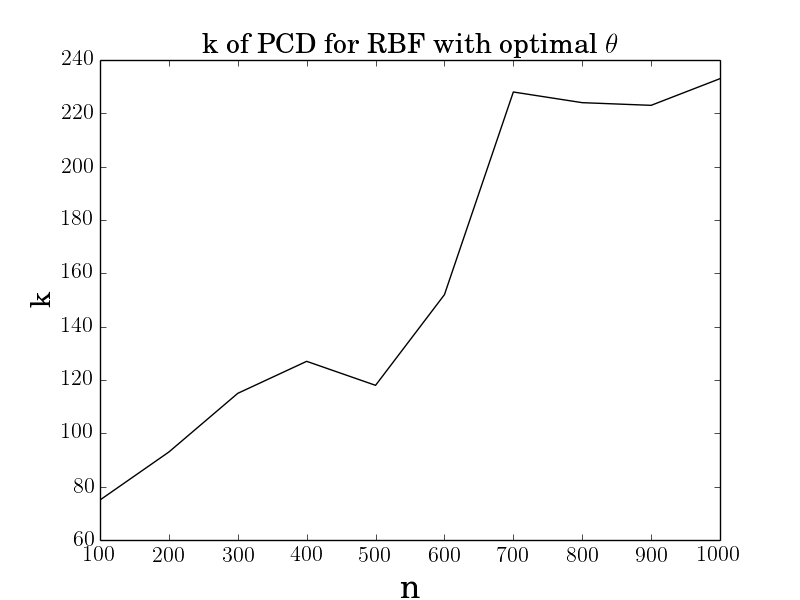}  \\   
\caption{Error of RBF approximations of $f_2$ along $n$ with optimal $\theta$ (left), and the associated $k$ of RBF-PCD (right)}
 \label{fig:error_varied_n_2D}
 \end{figure}
 
   \begin{figure}[htbp!]  
  \centering 
  \hspace{-0.5cm}
 \includegraphics[width=7.5cm]{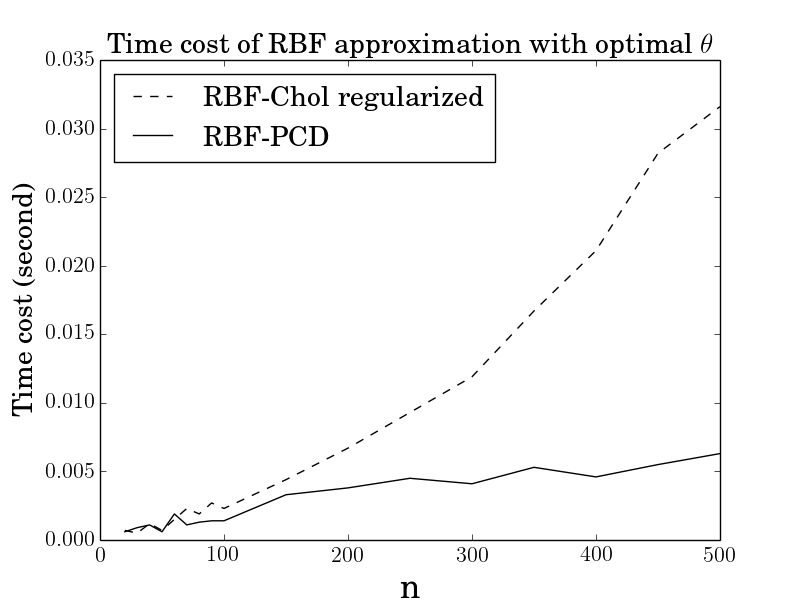}   
  \hspace{-0.2cm}
  \includegraphics[width=7.5cm]{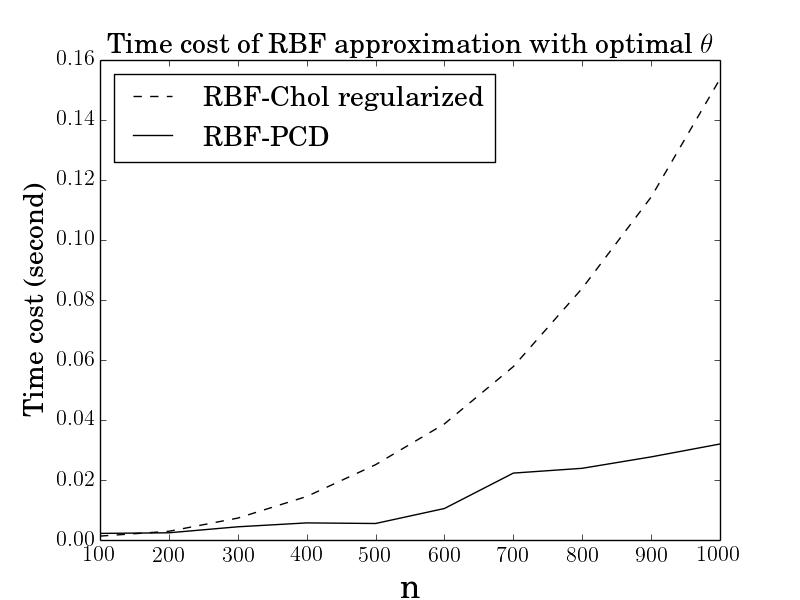}   
\caption{Time cost of RBF approximation of $f_1$ (left) and $f_2$ (right) with optimal $\theta$}
 \label{fig:time_varied_n}
 \end{figure}
 
\section{Summary} \label{sec:summary}
A new pivoted Cholesky decomposition (PCD) algorithm is proposed by tuning the cross approximation (CA) to symmetric positive semidefinite (SPSD) matrices.  The new algorithm gives a sharp bound of entry-wise error. As a by-product a diagonally pivoted CA algorithm is developed for efficient low rank approximation of SPSD matrices. PCD is proved  a valid solver of ill-conditioned systems with complexity  $\C O(k^2n)$,  $k$ is the rank of the system.    

The efficiency advantage of the PCD algorithm is numerically manifested in two  applications to radial basis function approximations.  And the diagonally pivoted CA algorithm is shown to greatly speed up eigen-decomposition of a large covariance matrix in an uncertainty quantification problem.

  \bibliographystyle{plain_only_initials_in_first_name}
\pagebreak

\bibliography{my_bib_library}

\end{document}